\documentclass[12pt, reqno]{amsart}
\usepackage{amsmath, amsthm, amscd, amsfonts, amssymb, graphicx, color, booktabs, multirow, rotating}
\newtheorem{theorem}{Theorem}[section]

\newtheorem{proposition}[theorem]{Proposition}

\theoremstyle{definition}
\newtheorem{definition}[theorem]{Definition}

\theoremstyle{remark}
\newtheorem{remark}[theorem]{Remark}
\numberwithin{equation}{section}

\usepackage{listings}

\begin{document}
\setcounter{page}{1}

\title{The $a$-number of $y^n=x^m+x$ over finite fields}

\author{Sepideh Farivar$^{2}$, Behrooz Mosallaei$^{1,\dag}$, Farzaneh Ghanbari$^{3}$ and Vahid Nourozi$^{1,\star}$}
\address{$^{1}$ The Klipsch School of Electrical and Computer Engineering, New Mexico State University,
Las Cruces, NM 88003 USA}
\email{behrooz@nmsu.edu$^{\dag}$, nourozi@nmsu.edu$^{\star}$}
\address{$^{2}$ Department of Computer Science, University of Nevada, Las Vegas, 4505 S. Maryland Parkway, Las Vegas, NV, 89154-4030, USA
}
\email{Farivar@unlv.nevada.edu}
\address{$^{1}$ Department of Pure Mathematics, Faculty of Mathematical Sciences,
 Tarbiat Modares University, P.O.Box:14115-134, Tehran, Iran.}
\email{behrooz@nmsu.edu$^{\dag}$, nourozi@nmsu.edu$^{\star}$}
\thanks{$^{\star}$Corresponding author}





\keywords{$a$-number; Cartier operator; Super-singular Curves; Maximal Curves.}

\begin{abstract}
This paper presents a formula for $a$-number of certain maximal curves characterized by the equation $y^{\frac{q+1}{2}} = x^m + x$ over the finite field $\mathbb{F}_{q^2}$. $a$-number serves as an invariant for the isomorphism class of the $p$-torsion group scheme. Utilizing the action of the Cartier operator on $H^0(\mathcal{X}, \Omega^1)$, we establish a closed formula for $a$-number of $\mathcal{X}$.

\end{abstract}
 \subjclass{11G20, 11M38, 14G15, 14H25}

\maketitle
\section{Introduction}

Consider $\mathcal{X}$ as a geometrically irreducible, projective, and non-singular algebraic curve defined over a finite field $\mathbb{F}{\ell}$ of order $\ell$. Set $\mathcal{X}(\mathbb{F}{\ell})$ denotes the $\mathbb{F}{\ell}$-rational points of $\mathcal{X}$. A fundamental problem in studying curves over finite fields is determining the size of $\mathcal{X}(\mathbb{F}{\ell})$. The Hasse-Weil bound provides a foundational result, stating that:
$$\mid \# \mathcal{X}(\mathbb{F}_{\ell}) - (\ell +1) \mid \leq 2g \sqrt{\ell},$$
where $g = g(\mathcal{X})$ represents genus $\mathcal{X}$.

The curve $\mathcal{X}$ is called maximal over $\mathbb{F}_{\ell}$ if the number of elements of $\mathcal{X}(\mathbb{F}_{\ell})$ satisfies
$$\# \mathcal{X}(\mathbb{F}_{\ell})= \ell + 1 + 2g \sqrt{\ell}.$$
We only consider the maximal curves of the positive genus; and hence, $\ell$ will always be a square, say, $\ell = q^2$.

In \cite{ih}, Ihara demonstrated that if a curve $\mathcal{X}$ is maximal over $\mathbb{F}_{q^2}$, then it's
$$g \leq g_1:= \dfrac{q(q-1)}{2}.$$

In \cite{tor} authors showed that

$$\qquad  \mbox{either} \qquad g \leq g_2 := \lfloor \frac{(q-1)^2}{4} \rfloor \qquad  \mbox{or} \qquad  g_1=\frac{q(q-1)}{2}. $$
Ruck and Stichtenoth \cite{stir} showed that, up to  $\mathbb{F}_{q^2}$-isomorphism, there is just one maximal curve over $\mathbb{F}_{q^2}$ of genus  $\frac{q(q-1)}{2}$, namely the so-called Hermitian curve over $\mathbb{F}_{q^2}$ which can be defined by the affine equation
 \begin{equation}\label{xxx1}
y^{q}+y=x^{q+1}.
 \end{equation}

Let $n,m \geq 2$ be integers such that $\mbox{gcd}(n,m) = 1, \mbox{gcd}(q, n) = 1$ and $\mbox{gcd}(q,m-1) = 1$ where $q = p^s$ for $s \geq 1$, and let $\mathcal{X}$ be the non-singular model over $\mathbb{F}_{q^2}$ of the plane affine curve
 \begin{equation}
y^{n}=x^{m}+x.
 \end{equation}
Note that $\mathcal{X}$ is a Hermitian curve over $\mathbb{F}_{q^2}$ if $n = q +1$ and $m = q$. And $\mathcal{X}$ has genus
 \begin{equation*}
g(\mathcal{X})=\dfrac{(m-1)(n-1)}{2}.
 \end{equation*}
In this paper we suppose that $n =\frac{q+1}{2}$ and $m = 2, 3$ or $m = p^b$ where $b$
divides $s$. Tafazolian and Torres in \cite{tafazol} proved that the curve $\mathcal{X}$ with the above
condition is Maximal curve over $\mathbb{F}_{q^2}$.

 Consider multiplication by $p$-morphism $[p]: X \rightarrow X$ which is a finite flat morphism of degree $p^{2g}$.  The factor is $[p]=V \circ F$. Here, $F: X \rightarrow X^{(p)}$   is the relative Frobenius morphism coming from the $p$-power map on the
structure sheaf; and  the Verschiebung morphism  $V: X^{(p)} \rightarrow X$ is the dual of $F$.  The kernel of multiplication by $p$ on $X$, is defined by the group of $X[p]$.
The important invariant  is the $a$-number $a(\mathcal{X})$ of curve $\mathcal{X}$ defined by
$$a(\mathcal{X})=\mbox{dim}_{\mathbb{\bar{F}}_p} \mbox{Hom}(\alpha_{p}, X[p]),$$
where $\alpha_p$ is the kernel of the Frobenius endomorphism in the group scheme $\mbox{Spec}(k[X]/(X^p))$. Another definition for the $a$-number is
$$ a(\mathcal{X}) = \mbox{dim}_{\mathbb{F}_p}(\mbox{Ker}(F) \cap \mbox{Ker}(V)).$$


 A few results on the rank of the Carteir operator (especially $a$-number) of curves is introduced by Kodama and Washio \cite{13}, González \cite{8}, Pries and Weir \cite{17}, Yui \cite{Yui} and Montanucci and Speziali \cite{maria}. Besides that, Vahid talked about the rank of the Cartier of the maximal curves in \cite{esfahan,aut}, the hyperelliptic curve in \cite{shiraz}, the maximal function fields of $\mathbb{F}_{q^2}$ in \cite{misori,behrooz}, the Picard Curve in \cite{picard} and, the Artin-Schreier curve and application of the Cartier operator in coding theory \cite{vahidd}. Vahid also wrote about these topics in his \cite{phd} dissertation.

In this study, we determine the $a$-number of certain maximal curves. In the case $g=g_1$, the $a$-number of Hermitian curves were computed by Gross \cite{10}.  Here, we compute $a$-number of maximal curves over $\mathbb{F}_{q^2}$  with genus $g=g_2$  for infinitely many values of $q$.


\section{The Cartier operator}
Let $k$ be an algebraically closed field of the characteristic $p>0$.
Let $\mathcal{X}$ be a
curve defined over $k$.
The Cartier operator is a $1/p$-linear operator acting on the sheaf $\Omega^1:=\Omega^1_{\mathcal{X}}$ of differential forms on a curve $\mathcal{X}$ with positive characteristic $p>0$.

 Let $K=k(\mathcal{X})$ represent the function field of curve $\mathcal{X}$ of genus $g$, defined over  $k$. defined over an algebraically closed field $K$ as an element $x \in K \setminus K^p$.

\begin{definition}
  (Cartier Operator). Let $\omega \in  \Omega_{K/K_q}$. There exist $f_0,\cdots, f_{p-1}$ such
that $\omega = (f^p_0 + f^p_1x +\cdots + f^p_{p-1}x^{p-1})dx$. The Cartier operator $\mathfrak{C}$ is defined by
$$\mathfrak{C}(\omega) := f_{p-1}dx.$$
This definition does not depend on the choice of $x$ (see \cite[Proposition 1]{100}).
\end{definition}
We refer the reader to \cite{100,30,40,150} for proofs of the following statements.

\begin{proposition}
  (Global Properties of $\mathfrak{C}$). For all $\omega \in  \Omega_{K/K_q}$ and all $f \in K$,

  \begin{itemize}
    \item $\mathfrak{C}(f^p\omega) = f\mathfrak{C}(\omega)$;
    \item $\mathfrak{C}(\omega) = 0 \Leftrightarrow \exists h \in K, \omega = dh$;
    \item $\mathfrak{C}(\omega) = \omega \Leftrightarrow \exists h \in K, \omega = dh/h$.
  \end{itemize}
\end{proposition}

\begin{remark}\label{remark}
\rm{ Moreover, one can easily show
that
\begin{equation*}
\mathfrak{C}(x^j dx) = \left\{
\begin{array}{ccc}
    0 & \mbox{if}&  \hspace{.4cm} p \nmid j+1  \\
    x^{s-1}dx & \mbox{if} &  \hspace{.4cm} j+1=ps.
\end{array} \right.
\end{equation*}}

\end{remark}

If $\mbox{div}(\omega)$ is effective, then differential $\omega$ is holomorphic. The set $H^0(\mathcal{X}, \Omega ^1)$ of holomorphic differentials is a $g$-dimensional $k$-vector subspace of $\Omega^1$ such that $\mathfrak{C}(H^0(\mathcal{X}, \Omega^1)) \subseteq H^0(\mathcal{X}, \Omega^1)$. If $\mathcal{X}$ is a curve, then $a$-number of $\mathcal{X}$ equals the dimension of the kernel of the Cartier operator $H^0(\mathcal{X}, \Omega^1)$ (or equivalently, the dimension of the space of exact holomorphic differentials on $\mathcal{X}$) (see \cite[5.2.8]{14}).

The following theorem is based on Gorenstein \cite[Theorem 12]{9}.
\begin{theorem}\label{2.2}
  A differential $\omega \in \Omega^1$ is holomorphic if and only if it is of the form $(h(x, y)/F_y)dx$, where $H:h(X, Y) =0$ is a canonical adjoint.
\end{theorem}

\begin{theorem}\cite{maria}\label{2.3}
  With the above assumptions,
\begin{equation*}
\mathfrak{C}(h\dfrac{dx}{F_{y}}) = (\dfrac{\partial^{2p-2}}{\partial x^{p-1}\partial y^{p-1}}(F^{p-1}h))^{\frac{1}{p}}\dfrac{dx}{F_{y}}
\end{equation*}
for any $h \in K(\mathcal{X})$.
\end{theorem}

The differential operator $\nabla$ is defined by
$$\nabla = \dfrac{\partial^{2p-2}}{\partial x^{p-1} \partial y^{p-1}},$$
has the property
\begin{equation}\label{123321}
  \nabla(\sum_{i,j} c_{i,j}X^iY^j)= \sum_{i,j} c_{ip+p-1,jp+p-1}X^{ip}Y^{jp}.
\end{equation}

\section{The $a$-number of Curve $\mathcal{X}$}\label{202}

In this section, we consider the curve $\mathcal{X}$ as follows: $y^{\frac{q+1}{2}}=x^{m}+x$ of genus $g(\mathcal{X})= \frac{(q-1)(m-1)}{4}$, with $q=p^s$ and $p>3$ and $m=2, 3$ or $m=p^s$ over $\mathbb{F}_{q^2}$. From Theorem \ref{2.2}, one can find a basis for the space $H^0(\mathcal{X}, \Omega^1)$ of holomorphic differentials on $\mathcal{X}$, namely
$$\mathcal{B} = \{ x^iy^jdx \mid 1 \leq \frac{q+1}{2}i+ mj \leq g \}.$$

 \begin{proposition}\label{111}
   The rank of the Cartier operator $\mathfrak{C}$ on the curve $\mathcal{X}$ equals the number of pairs $(i, j)$ with $\frac{q+1}{2}i+ mj \leq g$ such that the system of congruences mod $p$
\begin{equation}\label{12}
\Bigg\{
             \begin{array}{c}
              km + h - k + j \equiv 0,\\
              (p - 1 - h)(\frac{(q+1)}{2}) + i \equiv p-1, \\
             \end{array}
\end{equation}
has solution $(h, k)$ for $0 \leq h \leq \frac{p-1}{2}, 0 \leq k\leq h$.
 \end{proposition}
\begin{proof}
From Theorem \ref{2.3}, $\mathfrak{C}((x^iy^j/F_y)dx) =(\nabla(F^{p-1}x^iy^j))^{1/p}dx/F_y$. So, we apply the differential operator $\nabla$ to
\begin{equation}\label{333}
  (y^{\frac{q+1}{2}}-x-x^{m})^{p-1}x^iy^j = \sum_{h=0}^{p-1}\sum_{k=0}^{h} (^{p-1}_h)(^{h}_{k})(-1)^{h-k}x^{ (p - 1 - h)(\frac{(q+1)}{2}) + i}y^{km + h - k + j}
\end{equation}
for each $i, j$ such that $\frac{q+1}{2}i+ mj \leq g$.

From the Formula (\ref{123321}), $\nabla(y^{\frac{q+1}{2}}-x-x^{m})^{p-1}x^iy^j \neq 0$ if and only if for some $(h, k)$, with $0 \leq h \leq \frac{p-1}{2}$ and $0 \leq k\leq h$, satisfies both the following congruences mod $p$:
\begin{equation}\label{444}
\Bigg\{
             \begin{array}{c}
              km + h - k + j \equiv 0,\\
              (p - 1 - h)(\frac{(q+1)}{2}) + i \equiv p-1. \\
             \end{array}
\end{equation}
Take $(i, j) \neq(i_0, j_0)$ in this situation both $\nabla(y^{\frac{q+1}{2}}-x-x^{m})^{p-1}x^iy^j$ and $\nabla(y^{\frac{q+1}{2}}-x-x^{m})^{p-1}x^{i_0}y^{j_0}$ are nonzero. We claim that these are linearly independent of $k$. To show independence, we prove that, for each $(h, k)$ with $0 \leq h \leq p -1$ and $0 \leq k\leq h$ there is no $(h_0, k_0)$ with $0 \leq h_0\leq p -1$ and $0 \leq k_0\leq h_0$ such that
\begin{equation}\label{555}
  \Bigg\{
             \begin{array}{c}
             km + h - k + j = k_0m + h_0 - k_0 + j_0,\\
              (p - 1 - h)(\frac{(q+1)}{2}) + i = (p - 1 - h_0)(\frac{(q+1)}{2}) + i_0.\\
             \end{array}
\end{equation}
If $h=h_0$, then $j \neq j_0$ by $i=i_0$ from the second equation; thus, $k \neq k_0$. It is assumed $k >k_0$. Then $j -j_0=(q-1)(k -k_0)>q-1$, a contradiction as $j - j_0 \leq \frac{(q-1)(m-1)}{4q}$. Similarly, if $k=k_0$, then $h \neq h_0$ by $(i, j) \neq(i_0, j_0)$. It is assumed $h>h_0$. Then $i-i_0 = \frac{q+1}{2}(h-h_0)>\frac{q+1}{2}$, a contradiction as $i-i_0 \leq \frac{(q-1)(m-1)}{2(q+1)}$.
\end{proof}

For convenience, we signify the matrix representing the $p$-th power of the Cartier operator $\mathfrak{C}$ with $A_s:=A(\mathcal{X})$ on the curve $\mathcal{X}$ for the basis $\mathcal{B}$, where $q=p^s$.
\begin{theorem}\label{thex}
  If $q =p^s$ for $s \geq 1$, $p>3$ and $m=2, 3$ or $m=p^s$, then the $a$-number of the curve $\mathcal{X}$ equals
  $$\dfrac{m-1}{20}(2q+3p-5),$$
  where $\mbox{gcd}(\frac{q+1}{2},m)=1$ and $\mbox{gcd}(p,m-1)=1$.
\end{theorem}

\begin{proof}
First, we prove that if $q=p^s, s \geq 1$ and even, then $\mbox{rank}(A_s) =\dfrac{3}{20}(q-p)(m-1)$. In this case, $\frac{q+1}{2}i+ mj \leq g$ and System (\ref{12}) mod $p$ reads
  \begin{equation}\label{14}
\Bigg\{
             \begin{array}{c}
              km+h-k+j \equiv 0,\\
              -\frac{h}{2} - \frac{1}{2} + i \equiv p-1. \\
             \end{array}
\end{equation}
First assume that $s=1$, for $q =p$, we have $\frac{p+1}{2}i+ mj \leq g$ and System (\ref{14}) becomes
  \begin{equation*}
\Bigg\{
             \begin{array}{c}
              j = k-h-km,\\
              i = p + \frac{h}{2} - \frac{1}{2},\\
             \end{array}
\end{equation*}
In this case, $\frac{p+1}{2}i+ mj \leq g$ and $m=p^s$ that is, $\frac{h(p-4p^s+1)}{4} + kp^s \leq \frac{-2p^2 -p(1-p^s)-p^s+1}{4}$ and $h \geq \frac{-2p^2 -p(1-p^s)-p^s+1}{p+1-4p^s}$ thus, $h > \frac{p-1}{2}$ is a contradiction by Proposition \ref{111}. For $m=2,3$ is trivial. Consequently, there is no pair $(i, j)$ for which the system above admits a solution $(h, k)$. Therefore, $\mbox{rank}(A_1) =0$.

Let $s=2$, so $q =p^2$. For $\frac{p^2+1}{2}i+ mj \leq g$, the above argument still holds. Therefore, $\frac{(p-1)(m-1)}{4} + 1 \leq \frac{p^2+1}{2}i+ mj \leq \frac{(p^2-1)(m-1)}{4}$ and our goal is to determine for which $(i, j)$ there is a solution $(h, k)$ of the system mod $p$
  \begin{equation*}
\Bigg\{
             \begin{array}{c}
             km+h - k +j \equiv 0,\\
              -\frac{h}{2} - \frac{1}{2} + i \equiv p-1. \\
             \end{array}
\end{equation*}
Take $l, t \in Z^+_0$ so that
  \begin{equation*}
\Bigg\{
             \begin{array}{c}
              j=lp+k-h+km,\\
              i= tp+p+\frac{h}{2} - \dfrac{1}{2}.\\
             \end{array}
\end{equation*}

In this situation, $i < \frac{2g}{p^2+1}=\frac{(p^2-1)(m-1)}{2(p^2+1)}$, so $tp+p+\frac{h}{2}-\frac{1}{2} \leq \frac{(p^2-1)(m-1)}{2(p^2+1)}$. Then $t\leq \frac{(p^2-1)(m-1)}{2(p^2+1)}$. And $j < \frac{(p^2-1)(m-1)}{4m}$, so $lp + k-h<\frac{(p^2-1)(m-1)}{4m}$, Then $l<\frac{(p^2-1)(m-1)}{4m}$. Thus, we can say that $ l \leq \frac{3p(p-1)}{5}$, and $t \leq \frac{m-1}{4}$. Thus, $\frac{3(p^2-p)(m-1)}{20}$ suitable values for $(i, j)$ were obtained, where $\mbox{rank}(A_2) =\frac{3(p^2-p)(m-1)}{20}$.


 For $s \geq 3$, $\mbox{rank}(A_s)$ equals $\mbox{rank}(A_{s-1})$ plus the number of pairs $(i, j)$ with $\frac{(p^{s-1} -1)(m-1)}{4}+1\leq \frac{q+1}{2}i+ mj \leq \frac{(p^{s} -1)(m-1)}{4}$ such that the system mod $p$
  \begin{equation*}
\Bigg\{
             \begin{array}{c}
             km+ h-k+j \equiv 0,\\
              -\frac{h}{2} - \frac{1}{2} + i \equiv p-1, \\
             \end{array}
\end{equation*}
has a solution. With our usual conventions on $l, t$, a computation shows that such pairs $(i, j)$ are obtained for $0 \leq l \leq \frac{(p^s-1)(m-1)}{4m}$ from which we have $\frac{3p^{s-1}(p-1)}{5} -1 \leq l \leq \frac{3p^{s-1}(p-1)}{5}$, and $0 \leq t\leq \frac{(p^s-1)(m-1)}{2(p^s+1)}$, and we have $\frac{m-1}{4}-1 \leq t \leq \frac{(m-1)}{4}$. In this case we have
$$\frac{3p^{s-1}(p-1)(m-1)}{20}$$
choices of $(h, k)$. Therefore we get
$$\mbox{rank}(A_s)= \mbox{rank}(A_{s-1})+ \frac{3p^{s-1}(p-1)(m-1)}{20}.$$

Our claim on the rank of $A_s$ follows by induction on $s$. Hence

  \begin{equation*}
  \begin{array}{ccccccc}
              a(\mathcal{X}) &=& \dfrac{(p^s-1)(m-1)}{4} - \dfrac{3(q-p)(m-1)}{20}  \\
              &=& \dfrac{(m-1)(2q+3p-5)}{20} \\
             \end{array}
\end{equation*}
\end{proof}

\paragraph*{\textbf{Acknowledgements.}}
This paper was written while Vahid Nourozi visited Unicamp (Universidade Estadual de Campinas) supported by TWAS/Cnpq (Brazil) with fellowship number $314966/2018-8$.


\bibliographystyle{amsplain}

\end{document}